\documentclass[12pt]{article}

\usepackage{comment}
\usepackage{pgf,tikz}
\usepackage{pgfplots}
\usetikzlibrary{arrows.meta}

\usepackage{amsmath,amssymb,amsthm,cite}
\usepackage[colorlinks,citecolor=red,urlcolor=blue,linkcolor=blue]{hyperref}
\usepackage{mathrsfs,dsfont}
\usepackage{graphicx}
\usepackage{enumitem}

\newtheorem{nummer}{ }%[section]
\newtheorem{thm}[nummer]{\bf Theorem}
\newtheorem{prop}[nummer]{\bf Proposition}
\newtheorem{lem}[nummer]{\bf Lemma}
\newtheorem{cor}[nummer]{\bf Corollary}

\newtheorem{rmk}{\bf Remark}
\newcommand{\rmksty}{\rm}
\newtheorem{alg}{\bf Algorithm}
\newcommand{\algsty}{\rm}

\newcommand{\ie} {\sl i.e.}

\newcommand{\dppp}{double-pytha\-potent pair}
\newcommand{\dppps}{double-pytha\-potent pairs}
\newcommand{\qppp}{quadratic pytha\-potent pair}
\newcommand{\qppps}{quadratic pytha\-potent pairs}
\newcommand{\pp}{pytha\-gorean pair}
\newcommand{\pps}{pytha\-gorean pairs}

\newcommand{\tor}[2]{\Z/#1\Z\times\Z/#2\Z}

\newcommand{\Q}{\mathds{Q}}

\newcommand{\Z}{\mathds{Z}}

\newcommand{\THM}{Theorem}

\newcommand{\LEM}{Lemma}

\newcommand{\tq}{\tilde q}

\newcommand{\tm}{\tilde m}
\newcommand{\tn}{\tilde n}
\newcommand{\tp}{\tilde p}
\newcommand{\tb}{\tilde b}
\newcommand{\ta}{\tilde a}

\newcommand{\C}{\Gamma}

\newcommand{\Cab}{\C_{a,b}}
\newcommand{\Caabb}{\C_{a^2,b^2}}

\makeatletter
\def\opargproof[#1]{\par\noindent {\bf #1 }}
 \makeatother

\definecolor{darkgreen}{rgb}{0,.6,0}

\begin{document}
\begin{center}
{\LARGE\bf Pairing Pythagorean Pairs}%\\[1.8ex] 

\medskip

{\small Lorenz Halbeisen}\\[1.2ex] 
{\scriptsize Department of Mathematics, ETH Zentrum,
R\"amistrasse\;101, 8092 Z\"urich, Switzerland\\ lorenz.halbeisen@math.ethz.ch}\\[1.8ex]
{\small Norbert Hungerb\"uhler}\\[1.2ex] 
{\scriptsize Department of Mathematics, ETH Zentrum,
R\"amistrasse\;101, 8092 Z\"urich, Switzerland\\ norbert.hungerbuehler@math.ethz.ch}
\end{center}

\hspace{5ex}{\small{\it key-words\/}: pythagorean pair, elliptic curve}

\hspace{5ex}{\small{\it 2020 Mathematics Subject 
Classification\/}: {\bf 11D72} 11G05%% 
% 11D72 Diophantine equations in many variables
% 11G05 Elliptic curves over global fields 

\begin{abstract}\noindent
{\small A pair $(a,b)$ of positive integers is a {\it \pp\/} if 
$a^2+b^2=\Box$ ({\ie}, \hbox{$a^2+b^2$} is a square). A pythagorean pair $(a,b)$ is called
a {\it\dppp\/} if there is another {\pp} $(k,l)$ such that $(ak,bl)$ is a {\pp},
and it is called a {\it\qppp\/} if there is another {\pp} $(k,l)$ which is not a multiple
of $(a,b)$, such that $(a^2k,b^2l)$ is a {\pp}. 
To each {\pp} $(a,b)$ we assign an elliptic curve $\Gamma_{a,b}$
with torsion group $\tor 2 4$, such that $\Gamma_{a,b}$ has positive rank if and only if 
$(a,b)$ is a {\dppp}. Similarly, to each {\pp} $(a,b)$ we assign an elliptic curve $\Gamma_{a^2,b^2}$
with torsion group $\tor 2 8$, such that $\Gamma_{a^2,b^2}$ has positive rank if and only if 
$(a,b)$ is a {\qppp}. Moreover, in the later case we obtain that every elliptic curve $\Gamma$
with torsion group $\tor 2 8$ is isomorphic to a curve of the form $\Gamma_{a^2,b^2}$, where
$(a,b)$ is a {\pp}. As a side-result we get that if $(a,b)$ is a {\dppp}, then there are 
infinitely many {\pps} $(k,l)$, not multiples of each other, such that
$(ak,bl)$ is a {\pp}; the analogous result holds for {\qppps}.
}
\end{abstract}
%%%%%%%%%%%%%%%%%%%%%%%%%%%%%%%%%%%%%%%%%%%%%%%%%%%%%%%%%%

\section{Introduction}

A pair $(a,b)$ of positive integers is a {\it \pp\/} if 
\hbox{$a^2+b^2$} is a square, denoted 
$a^2+b^2=\Box\,.$ A pythagorean pair $(a,b)$ is called
a {\it\dppp\/} if there is another {\pp} $(k,l)$ such that $(ak,bl)$ is a {\pp}, {\ie},
$$a^2+b^2=\Box\,,\qquad k^2+l^2=\Box\,,\qquad\text{and}\qquad (ak)^2+(bl)^2=\Box\,.$$
Notice that since for positive integers $a$ and $b$, $a^4+b^4$ is never a square
(see~\cite[Oeuvres, I, p.\,327; III, p.\;264]{fermat-oevres}, 
% https://gallica.bnf.fr/ark:/12148/bpt6k6213144d/f383.item.texteImage
% https://gallica.bnf.fr/ark:/12148/bpt6k62145354/f292.item.texteImage
and $(a^2,b^2)$ is never a {\pp}.
Furthermore, a pythagorean pair $(a,b)$ is called a {\it\qppp\/} 
if there is another {\pp} $(k,l)$ which is not a multiple
of $(a,b)$, such that $(a^2k,b^2l)$ is a {\pp}, {\ie},
$$a^2+b^2=\Box\,,\qquad k^2+l^2=\Box\,,\qquad\text{and}\qquad (a^2k)^2+(b^2l)^2=\Box\,.$$ 

To each {\pp} $(a,b)$ we assign the elliptic curve 
$$
\Cab:\hspace*{3ex} y^2\ =\ x^3+(a^2+b^2)x^2+a^2b^2x\,,
$$
and show that the curve $\Cab$ has torsion group isomorphic to $\tor 2 4$ and that
$(a,b)$ is a {\dppp} if and only if $\Cab$ has positive rank.
With the points of infinite order on the curve $\Cab$, we can generate
infinitely many {\pps} $(k,l)$, not multiples of each other, such that
$(ak,bl)$ are {\pps}.

Similarly, for each {\pp} $(a,b)$, the elliptic curve 
$$
\Caabb:\hspace*{3ex} y^2\ =\ x^3+(a^4+b^4)x^2+a^4b^4x\,,
$$
has torsion group isomorphic to $\tor 2 8$ and 
$(a,b)$ is a {\qppp} if and only if $\Caabb$ has positive rank.
Moreover, we can show that every elliptic curve $\Gamma$
with torsion group $\tor 2 8$ is isomorphic to a curve of the form $\Caabb$ for some
{\pp} $(a,b)$. 
Similar as above, with the points of infinite order on the curve $\Caabb$, we can generate
infinitely many {\pps} $(k,l)$, not multiples of each other, such that
$(a^2k,b^2l)$ are {\pps}.

\begin{rmk} {\rmksty The results are similar for congruent numbers: To each positive 
integer $A$ we can assign an elliptic curve 
$$\C_{A}:\hspace*{3ex} y^2\ =\ x^3-A^2x 
$$
with torsion group isomorphic to $\tor 2 2$ such that $A$ is a {\it congruent
number\/} if and only if $\C_{A}$ has positive rank. Moreover, with the points 
of infinite order on the curve $\C_{A}$, we can generate
infinitely rational triples $(r,s,t)$ such that $r^2+s^2=t^2$ and $\frac{rs}2=A$
(an elementary proof of this result is given in~\cite{Fermat}).}
\end{rmk}

%%%%%%%%%%%%%%%%%%%%%%%%%%%%%%%%%%%%%%%%%%%%%%%%%%%%%%%%%%%%%%%%%%%%%%%%%%%%%%%%%%%%%%%%%%%%

\section{Quadratic Pythapotent Pairs}

In this section we consider {\qppps}\,---\,this case is slightly easier 
than the case with {\dppps}.
First we show that the curve $\Caabb$ has torsion group isomorphic to $\tor 2 8$, and 
then we show how we obtain {\pps} $(k,l)$ from a point on $\Caabb$ with quadratic
$x$-coordinate such that  $(a^2k,b^2l)$ is a {\pp}.

\begin{prop} If $(a,b)$ is a {\pp}, then the elliptic curve $\Caabb$ has torsion group $\tor 2 8$.
Vice versa, if an elliptic curve $\C$ has torsion group $\tor 2 8$, then there exists a 
{\pp} $(a,b)$ such that $\C$ is isomorphic to $\Caabb$.
\end{prop}

\begin{proof}
Kubert~\cite[p.\,217]{Kubert} gives the following parametrization for elliptic curves
with torsion group $\tor 2 8$ (see also Rabarison~\cite[3.14]{Rabarison}): 
$$y^2+(1-c)xy-ey\ =\ x^3-ex^2$$
for $$\tau=\frac{\tm}{\tn}\,,\qquad d=\frac{\tau(8\tau+2)}{8\tau^2-1}\,,\qquad 
c=\frac{(2d-1)(d-1)}{d}\,,\qquad e=(2d-1)(d-1)\,.$$
After a rational transformation we obtain the curve 
$$y^2\ =\ x^3+\ta x^2+\tb x$$
with $$\ta=256 \tm^4 (2 \tm + \tn)^4 + (4 \tm^2 - (2 \tm + \tn)^2)^4\quad
\text{and}\quad\tb= 256 \tm^4 \tn^4 (2 \tm + \tn)^4 (4 \tm + \tn)^4\,.$$
Let $m:=\tm$ and $n:=\frac{2\tm+\tn}{2}$. Then we obtain the curve
$$y^2\ =\ x^3+2^8\bigl((2mn)^4 + (m^2 - n^2)^4)\bigr)x^2+2^{16}\bigl(
(2mn)^4\cdot(m^2 - n^2)^4\bigr)x\,,$$ which is, for $a:=m^2-n^2$ and $b:=2mn$, 
equivalent to the curve
$$\Caabb:\hspace{3ex}
y^2\ =\ x^3+(a^4 +  b^4)x^2+a^4b^4x\,.$$
Notice that by definition of $a$ and $b$, $(a,b)$ is a {\pp}.
   
For the other direction, recall that for every {\pp} $(a,b)$ we find 
positive integers $\lambda,m,n$ such that 
$m$ and $n$ are relatively prime and $\{a,b\}=\big{\{}\lambda(m^2-n^2),\lambda(2mn)\big{\}}$.
So, by the substitutions $\tm:=m$ and $\tn:=2(n-m)$, we see that every elliptic curve
$\Gamma$ with torsion group $\tor 2 8$ is isomorphic to a curve of the form 
$\Caabb$ for some {\pp} $(a,b)$. 
\end{proof}
   
\begin{rmk}
{\rmksty Let $a:=m^2-n^2$ and $b:=2mn$. 
If we replace $m$ and $n$ by $\bar{m}:=m+n$ and $\bar{n}:=m-n$, respectively, 
even though we obtain another {\pp} $(a',b')$, the corresponding elliptic curves
$\Caabb$ and $\C_{\bar{a}^2,\bar{b}^2}$ are equivalent.}
\end{rmk}

\begin{thm}\label{thm:main28} 
The {\pp} $(a,b)$ is a {\qppp} if and only if the elliptic curve $\Caabb$ 
has positive rank.
\end{thm}

In order to prove {\THM}\;\ref{thm:main28}, we first transform the curve
$\Caabb$ to a another curve on which we carry out our calculations.

\begin{lem}\label{lem:trans28} 
If $x_2$ is the $x$-coordinate of a rational point on
$\Caabb$, then $$x_0:=\frac{a^2b^2}{x_2}$$ is the $x$-coordinate of a rational point on
the curve $$y^2x\ =\ a^2b^2+(a^4+b^4)x+a^2b^2 x^2\,.$$
\end{lem}

\begin{proof}
We work with homogeneous coordinates $(x,y,z)$. 
Consider the following transformation:
$$
\begin{pmatrix}
x\\
y\\
z
\end{pmatrix}
:=
\begin{pmatrix}
0 & 0 & 1\\
0 & 1 & 0\\
\frac{1}{a^{\mathstrut 2}b^2} & 0 & 0
\end{pmatrix}
\cdot
\begin{pmatrix}
X\\
Y\\
Z
\end{pmatrix}
$$
The the point $(x,y,z)$ belongs to the homogenized curve $\Caabb$ if and only if
the point $(X,Y,Z)$ belongs to the curve  $Y^2X =a^2b^2Z^3+(a^4+b^4)XZ^2+a^2b^2 X^2Z$.
Hence, by dehomogenizing, we obtain the curve 
$y^2x =a^2b^2+(a^4+b^4)x+a^2b^2 x^2$, which is equivalent to $\Caabb$,
where the rational point $(x_2,y_2)$ belongs to $\Caabb$ if and only if there is 
a rational $y'$ such that $(x_0,y')$ belongs to $y^2x =a^2b^2+(a^4+b^4)x+a^2b^2 x^2$.
\end{proof}

Let $x_0=\frac{p^2}{q^2}$ be a rational square and assume that~$x_0$ 
is the $x$-coordinate of a rational point on
$y^2x=a^2b^2+(a^4+b^4)x+a^2b^2 x^2$. Then, by dividing through $x_0$ and clearing 
quadratic denominators we obtain
$$a^2b^2\cdot q^4+(a^4+b^4)\cdot p^2\cdot q^2+a^2b^2\cdot p^4\ =\ \Box\,,$$
and since
$$a^2b^2\cdot q^4+(a^4+b^4)\cdot p^2\cdot q^2+a^2b^2\cdot p^4\;=\;
(a^2q^2+b^2p^2)\cdot (a^2p^2+b^2q^2)\,,$$
this is surely the case when
\begin{equation}
a^2q^2+b^2p^2\;=\;\Box\qquad\text{and}\qquad
a^2p^2+b^2q^2\;=\;\Box\,.\label{eq:28}  %\tag{*}\label{eq:28}
\end{equation}

\begin{lem}\label{lem:main28}
Let $P=(x_1,y_1)$ be a rational point on $\Caabb$ and let~$x_2$ be
the $x$-coordinate of the point $2*P$. 
Then $x_0:=\frac{a^2b^2}{x_2}=\frac{p^2}{q^{\mathstrut 2}}$, 
where $p$ and $q$ {satisfy\/}~{\rm (\ref{eq:28})}.
\end{lem}

\begin{proof}
By Silverman and Tate~\cite[p.27]{SilvermanTate}, 
$$x_2=\frac{(x_1^2-B)^2}{(2y_1)^{\mathstrut 2}}\qquad\text{for $B:=a^4b^4$,}$$
and therefore
$$x_0\;=\;\frac{a^2b^2}{x_2}\;=\;\frac{a^2b^2(2y_1)^2}{(x_1^2-B)^{\mathstrut 2}}\;=\;
\frac{a^2b^2\bigl(4x_1^3+4 Ax_1^2+4 B x_1\bigr)}{(x_1^2-B)^{\mathstrut 2}}\;=\;
\frac{p^2}{q^{\mathstrut 2}}\qquad\text{for $A:=a^4+b^4$.}$$
Now, for $p$ and $q$ (with $a=m^2-n^2$ and $b=2mn$) we obtain
$$
a^2q^2+b^2p^2\;=\;a^2\bigl(a^4b^4+2b^4 x_1+x_1^2\bigr)^2\;=\;\Box
$$
and
$$
a^2p^2+b^2q^2\;=\;b^2\bigl(a^4b^4+2a^4 x_1+x_1^2\bigr)^2\;=\;\Box
$$
which completes the proof.
\end{proof}

The next result gives a relation between rational points on $\Caabb$ with
quadratic $x$-coordinate and {\pps} $(k,l)$ such that $(a^2 k,b^2 l)$ is a {\pp}.

\begin{lem}\label{lem:final28} Every {\pp} $(k,l)$ such that $(a^2 k,b^2 l)$ 
corresponds to a rational point on $\Caabb$ with quadratic $x$-coordinate,
and vice versa.
\end{lem}

\begin{proof}
Let $x_2=\Box$ be the $x$-coordinate of a rational point on
$\Caabb$. Then, by {\LEM}\;\ref{lem:main28}, 
$\frac{a^2b^2}{x_2}=\frac{p^2}{q^{\mathstrut 2}}$, 
where $p$ and $q$ satisfy~(\ref{eq:28}), {\ie},
$a^2q^2+b^2p^2=\Box$. So, $\frac{a^2}{b^{\mathstrut 2}}+
\frac{p^2}{q^{\mathstrut 2}}=\rho^2$ for some $\rho\in\Q$.
In other words, we have 
$$\Bigl(\frac ab\Bigr)^2+\Bigl(\frac pq\Bigr)^2\;=\;\rho^2\,,$$
which implies that $$\frac ab=\frac{2\rho t}{t^2+1}\qquad\text{and}\qquad
\frac pq=\frac{\rho (t^2-1)}{t^2+1}\qquad\text{for some $t\in\Q$.}$$
In particular, we have $$\rho\;=\;\frac{a\cdot(t^2+1)}{b\cdot(2t)}\,.$$
Now, since $a^2p^2+b^2q^2=\Box$, we have $\bigl(\frac ab\bigr)^2+\bigl(\frac qp\bigr)^2=\Box$, 
hence, $\frac{a^2}{b^2}+\frac{(t^2+1)^2}{\rho^{\mathstrut 2} (t^2-1)^2}=\Box$, which implies that
$$a^4\cdot(t^2-1)^2+b^4\cdot (2t)^2\;=\;\Box\,.$$ For $t=\frac rs$, we obtain
$$\frac{a^4\cdot(r^2-s^2)^2}{s^4}+\frac{b^4\cdot 4r^2}{s^2}\;=\;\Box\,,$$
which implies that 
$$a^4\cdot(r^2-s^2)^2+b^4\cdot(2rs)^2\;=\;\Box\,,$$ and for $k:=r^2-s^2$, $l:=2rs$, we finally
obtain $$(a^2k)^2+(b^2 l)^2\;=\;\Box\qquad\text{where $k^2+l^2=\Box$}\,,$$
which shows that $(a,b)$ is a {\qppp}.

Assume now that we find a {\pp} $(k,l)$ such that $(a^2 k,b^2 l)$ is a {\pp}. 
Without loss of generality we may assume that $k$ and $l$ are relatively prime.
Thus, we find relatively prime positive integers $r$ and $s$ such that 
$k=r^2-s^2$ and $l=2rs$. With $t:=\frac rs$, $a$, and $b$, we can compute $p$ and
$q$, and finally obtain a rational point on $\Caabb$ whose $x$-coordinate is a square.
\end{proof}

We are now ready for the

\begin{proof}[Proof of Theorem\;\ref{thm:main28}]
For every rational point $P$ on $\Caabb$ with quadratic $x$-coordinate
let $(k_P,l_P)$ be the corresponding {\pp}. 
By {\LEM}\;\ref{lem:final28} it is enough to show that $(k_P,l_P)$
is a multiple of $(a,b)$ if and only if $P$ is a torsion point.
Notice that if $P$ is a point of infinite order, then for every
integer $i$, $2i*P$ is a rational point on $\Caabb$ with
quadratic $x$-coordinate, and not all of the corresponding {\pp}
$(k_{2i*P},l_{2i*P}$ can be multiples of $(a,b)$.

Let us consider the $x$-coordinates of the torsion points on 
the curve $\Caabb$. For simplicity, we consider the $16$
torsion points on the equivalent curve 
$$y^2=\frac{a^2b^2}{x}+(a^4+b^4)+a^2b^2 x\,.$$
The two torsion points at infinity are $(0,1,0)$ (which is the neutral
element of the group) and $(1,0,0)$ (which is a point of order~$2$). 
The other two points of order~$2$ are $(-\frac{a^2}{b^2},0)$ and
$(-\frac{b^2}{a^2},0)$, and the two points of order~$4$ are 
$\bigl(1,\pm(a^2+b^2)\bigr)$. The $x$-coordinates of the other~$10$
torsion points are $\frac{m(m+n)}{n(m-n)}$, $\frac{n(m-n)}{m(m+n)}$, 
$-\frac{m(m-n)}{n(m+n}$, $-\frac{n(m+n)}{m(m-n)}$, and $-1$.
Obviously, $-1$, $-\frac{a^2}{b^2}$, and $-\frac{b^2}{a^2}$ are not
squares of rational numbers. Furthermore, $0$ would lead to $p=0$, $q=1$, $t=1$,
$r=1$, $s=0$, $k=1$ and $l=0$, and therefore, $(k,l)$ is not a {\pp}.
If $\frac{m(m+n)}{n(m-n)}=\Box$, then, by multiplying with $n^2(m-n)^2$,
also $mn(m^2-n^2)=\Box$, which would imply that $A:=mn(m^2-n^2)$ is a 
congruent number with $A=\Box$. But this is impossible, since otherwise $1$ would 
be a congruent number, which is not the case (see also~\cite[Oeuvres, I, p.\,340]{fermat-oevres}
% https://gallica.bnf.fr/ark:/12148/bpt6k6213144d/f396.item.texteImage
or~\cite[p.\,163]{zeuthen} for  an annotated version of Fermat's proof).
% https://archive.org/details/geschichtederma00meyegoog/page/n175/mode/2up
Similarly, one can show 
that also $\frac{n(m-n)}{m(m+n)}$, $-\frac{m(m-n)}{n(m+n}$ and 
$-\frac{n(m+n)}{m(m-n)}$ cannot be squares. Thus, the only value of
$x$-coordinates of torsion points on the curve $\Caabb$ which is a square
is \hbox{$x=1$.} This leads to $k=2b$ and $l=2a$, {\ie}, to the {\pp} $(2b,2a)$, which
is a multiple of $(a,b)$\,---\,notice that for $c:=a^2+b^2$, $(2a^2b)^2+
(2ab^2)^2=(2abc)^2$.
\end{proof}

\begin{cor}
If $(a,b)$ is a {\qppp}, then there are 
infinitely many {\pps} $(k,l)$, not multiples of each other, such that
$(ak,bl)$ is a {\pp}.
\end{cor}

\begin{proof}
By {\THM}\;\ref{thm:main28}, there exists a point~$P$
on $\Caabb$ of infinite order. Now, for every
integer $i$, $2i*P$ is a rational point on $\Caabb$ with
quadratic $x$-coordinate, and each of the corresponding {\pps}
$(k_{2i*P},l_{2i*P})$ can be a multiple of just finitely many other
such {\pp}. Thus, there are infinitely many integers~$j$,
such that the {\pps} $(k_{2j*P},l_{2j*P})$ are not 
multiples of each other.
\end{proof}

\begin{alg}{\algsty
The following algorithm decribes how to construct {\pps} $(k,l)$ from rational
points on $\Caabb$ of infinite order.
\begin{itemize}
\item Let $P$ be a rational point on $\Caabb$ of infinite order and
let $x_2$ be the $x$-coordinate of $2*P$.
\item Let $p$ and $q$ be relatively prime positive integers such
that $$\frac{q}{p}\;=\;\frac{\sqrt{x_2}}{ab}\,.$$
\item Let $r$ and $s$ be relatively prime positive integers such
that $$\frac rs\;=\;\frac{bp+\sqrt{a^2q^2+b^2p^2}}{aq}\,.$$
\item Let $k:=r^2-s^2$ and let $l:=2rs$.
\end{itemize}
Then $(a^2 k,b^2 l)$ is a {\pp}. 
}
\end{alg}

%%%%%%%%%%%%%%%%%%%%%%%%%%%%%%%%%%%%%%%%%%%%%%%%%%%%%%%%%%%%%%%%%%%%%%%%%%%%%%%%%%%%%%%%%%%%

\section{Double-Pythapotent Pairs}

Below we consider {\dppps}. As above, we first show
that the curve $\Cab$ has torsion group isomorphic to $\tor 2 4$, and 
then we show how we obtain {\pps} $(k,l)$ from a point on $\Cab$ with quadratic
$x$-coordinate such that $(ak,bl)$ is a {\pp}. Since the calculations are similar,
we shall omit the details.

\begin{prop} If $(a,b)$ is a {\pp}, then the elliptic curve 
$$
\Cab:\hspace*{3ex} y^2\ =\ x^3+(a^2+b^2)x^2+a^2b^2x\,,
$$
has torsion group $\tor 2 4$.
\end{prop}

\begin{proof}
Kubert~\cite[p.\,217]{Kubert} gives the following parametrization for elliptic curves
with torsion group $\tor 2 4$: $$y^2+xy-ey\ =\ x^3-ex^2$$
for $$e=v^2-\tfrac 1{16}\qquad\text{where $v\neq 0,\;\pm\tfrac 14$}\,.$$
After a rational transformation we obtain the curve 
$$y^2\ =\ x^3+\ta x^2+\tb x$$
with $$\ta=2\cdot(16v^2+1)\quad
\text{and}\quad\tb=(16v^2-1)^2\,.$$
For $v=\frac pq$, $a=m^2-n^2$, $b=2mn$, let $p:=\frac 18(a-b)$ and $q:=\frac 12(a+b)$.
Then the curve $y^2+xy-ey=x^3-ex^2$ is equivalent to the curve 
$$\Cab:\hspace{3ex}
y^2\ =\ x^3+(a^2 +  b^2)x^2+a^2b^2x\,.$$
\end{proof}
   
\begin{rmk}
{\rmksty Notice that there are $p$ and $q$ which are not of the above form,
which implies that there are curves with torsion group $\tor 2 4$ which 
are {\it not\/} equivalent to some curve~$\Cab$.}
\end{rmk}

\begin{thm}\label{thm:main24} 
The {\pp} $(a,b)$ is a {\dppp} if and only if the elliptic curve $\Cab$ 
has positive rank.
\end{thm}

In order to prove {\THM}\;\ref{thm:main24}, we again transform the curve
$\Cab$ to a another curve on which we carry out our calculations.

\begin{lem}\label{lem:trans24} 
If $x_2$ is the $x$-coordinate of a rational point on
$\Cab$, then $$x_0:=\frac{ab}{x_2}$$ is the $x$-coordinate of a rational point on
the curve $$y^2x\ =\ ab+(a^2+b^2)x+ab x^2\,.$$
\end{lem}

\begin{proof}
We can just follow the proof of {\LEM}\;\ref{lem:trans28}, using
the transformation
$$
\begin{pmatrix}
0 & 0 & 1\\
0 & 1 & 0\\
\frac{1}{a\mathstrut b} & 0 & 0
\end{pmatrix}\,.
$$
\end{proof}

Let $x_0=\frac{p}{q}$ be the $x$-coordinate of a rational point on
$y^2x=ab+(a^2+b^2)x+ab x^2$, where $q=\tq^2$ and 
$p=ab\cdot\tp^2$ for some integers $\tq,\tp$. Then $$ab\cdot y^2\cdot\frac pq\;=\;
ab\cdot y^2\cdot\frac{ab\tp^2}{\tq^{\mathstrut 2}}\;=\;
y^2\cdot\Bigl(\frac{ab\cdot\tp}{\tq}\Bigr)^2\;=\;\Box\,.$$
Therefore, $$ab\cdot\Bigl(ab+(a^2+b^2)\cdot\tfrac pq+ab\cdot\tfrac{p^2}{q^{\mathstrut 2}}
\Bigr)\;=\;\Box\,,$$
and by clearing quadratic denominators we obtain
$$ab\cdot\bigl(aq+bp\bigr)\cdot\bigl(ap+bq\bigr)\;=\;\Box\,,$$
which is surely the case when
\begin{equation}
a\cdot(aq+bp)\;=\;\Box\qquad\text{and}\qquad
b\cdot(ap+bq)\;=\;\Box\,.\label{eq:24}  %\tag{\ding{128}}\label{eq:24}
\end{equation}

\begin{lem}\label{lem:main24}
Let $P=(x_1,y_1)$ be a rational point on $\Cab$ and let~$x_2$ be
the $x$-coordinate of the point $2*P$. 
Then $x_0:=\frac{ab}{x_2}=\frac{p}{q}$, where 
$q=\tq^2$ and $p=ab\cdot\tp^2$ for some integers $\tq,\tp$ 
and $p$ and $q$ {satisfy\/}~{\rm (\ref{eq:24})}.
\end{lem}

\begin{proof}
By Silverman and Tate~\cite[p.27]{SilvermanTate}, 
$$x_2=\frac{(x_1^2-B)^2}{(2y_1)^{\mathstrut 2}}\qquad\text{for $B:=a^4b^4$,}$$
and therefore
$$x_0\;=\;\frac{ab}{x_2}\;=\;
\frac{ab\bigl(4x_1^3+4 Ax_1^2+4 B x_1\bigr)}{(x_1^2-B)^{\mathstrut 2}}\;=\;
\frac{p}{q}\qquad\text{for $A:=a^4+b^4$.}$$
So, $q=\Box$ and $p=ab\cdot\tp^2$ for some integer $\tp$.

Now, for $x_1=\frac uv$ and $x_0=\frac pq$ (with $a=m^2-n^2$ and $b=2mn$) we obtain
$$
a\cdot(aq+bp)\;=\frac{1}{v^{\mathstrut 4}}\Bigl(
a^2\cdot\bigl(a^2b^2v^2+u(u+2b^2v)\bigr)
\Bigr)^2\;=\;\Box
$$
and
$$
b\cdot(ap+bq)\;=\frac{1}{v^{\mathstrut 4}}\Bigl(
b^2\cdot\bigl(a^2b^2v^2+u(u+2a^2v)\bigr)
\Bigr)^2\;=\;\Box
$$
which completes the proof.
\end{proof}

The next result gives a relation between rational points on $\Cab$ with
quadratic $x$-coordinate and {\pps} $(k,l)$ such that $(a^2 k,b^2 l)$ is a {\pp}.

\begin{lem}\label{lem:final24} Every {\pp} $(k,l)$ such that $(a^2 k,b^2 l)$  is a {\pp}
corresponds to a rational point on $\Cab$ with quadratic $x$-coordinate,
and vice versa.
\end{lem}

\begin{proof}
Let $x_2=\Box$ be the $x$-coordinate of a rational point on
$\Cab$. Then, by {\LEM}\;\ref{lem:main24}, 
$\frac{ab}{x_2}=\frac{ab\cdot f^2}{g^{\mathstrut 2}}$, 
where $p=ab\cdot f^2$ and $q=g^2$ satisfy~(\ref{eq:24}), {\ie},
$a^2g^2+a^2b^2f^2=\Box$. So, $\bigl(\frac{g}{f}\bigr)^2+
b^2=\rho^2$ for some $\rho\in\Q$ and $\bigl(\frac{g}{f}\bigr)^2+
a^2=\Box$. Let $\frac gf=\frac{2\rho t}{t^2+1}$ and 
$b=\frac{\rho (t^2-1)}{t^2+1}$. Then $\rho=\frac{b(t^2+1)}{t^2-1}$ and
$\frac gf=\frac{2 b t}{t^{\mathstrut 2}-1}$, which gives us
$$t=\frac{bf\pm\sqrt{g^2+b^2f^2}}{g}\,.$$
Since $$g^2+b^2f^2\;=\;q+\tfrac{b^2p}{ab}\;=\;q+\tfrac{bp}{a}\,,$$ 
by multiplying with $a^2$ we get $$a^2\cdot(g^2+b^2f^2)\;=\;
a^2\cdot q+ab\cdot p\;=\;a(aq+bp)\,.$$ Hence, by {\LEM}\;\ref{lem:main24},
$g^2+b^2f^2=\Box$ and therefore $t$ is rational, say $t=\frac rs$.
Finally, since $\bigl(\frac{g}{f}\bigr)^2+a^2=\Box$, we obtain
$$a^2\cdot(r^2-s^2)^2+b^2\cdot(2rs)^2\;=\;\Box\,,$$ and for $k:=r^2-s^2$, $l:=2rs$, we finally
get $$(ak)^2+(bl)^2\;=\;\Box\qquad\text{where $k^2+l^2=\Box$}\,,$$
which shows that $(a,b)$ is a {\dppp}.

Assume now that we find a {\pp} $(k,l)$ such that $(ak,bl)$ is a {\pp}. 
Without loss of generality we may assume that $k$ and $l$ are relatively prime.
Thus, we find relatively prime positive integers $r$ and $s$ such that 
$k=r^2-s^2$ and $l=2rs$. With $t:=\frac rs$, $a$, and $b$, we can compute $p$ and
$q$, and finally obtain a rational point on $\Cab$ whose $x$-coordinate is a square.
\end{proof}

We are now ready for the

\begin{proof}[Proof of Theorem\;\ref{thm:main24}]
For every rational point $P$ on $\Cab$ with quadratic $x$-coordinate
let $(k_P,l_P)$ be the corresponding {\pp}. 
By {\LEM}\;\ref{lem:final24} it is enough to show that no rational point
with quadratic $x$-coordinate has finite order.

Let us consider the $x$-coordinates of the torsion points on 
the curve $\Cab$. For simplicity, we consider the $8$
torsion points on the equivalent curve 
$$y^2=\frac{ab}{x}+(a^2+b^2)+ab x\,.$$
The two torsion points at infinity are $(0,1,0)$ (which is the neutral
element of the group) and $(1,0,0)$ (which is a point of order~$2$). 
The other two points of order~$2$ are $(-\frac{a}{b},0)$ and
$(-\frac{b}{a},0)$, and the four points of order~$4$ are 
$\bigl(1,\pm(a+b)\bigr)$ and $\bigl(-1,\pm(a-b)\bigr)$. 
Now, we have that none of the values 
$$\frac{1}{ab}\,,\qquad \frac{-1}{ab}\,,\qquad
\frac{-\frac ab}{ab}=-\frac 1{b^{\mathstrut 2}}\,,\qquad 
\frac{-\frac ba}{ab}=-\frac 1{a^{\mathstrut 2}}\,,$$ is a rational square.
For example, if $\frac{1}{ab}=\Box$, then $ab=\Box$, and since $b=2mn$, 
this implies that $ab=4\cdot\Box$. So, we have $\frac{ab}{2}=2\cdot\Box$, 
which is impossible (see~\cite[p.\;175]{Frenicle29}).
Thus, there is no {\pp} $(k,l)$ such that $(ak,bl)$ is a {\pp}.
\end{proof}

Similar as above, we get the following

\begin{cor}
If $(a,b)$ is a {\dppp}, then there are 
infinitely many {\pps} $(k,l)$, not multiples of each other, such that
$(ak,bl)$ is a {\pp}.
\end{cor}

\begin{rmk}{\rmksty
Let $(a,b)$ be a {\dppp} and let $(k_1,l_1)$ be a {\pp} such that 
$(ak_1,bl_1)$ is a {\pp}. Then $(k_1,l_1)$ is a {\dppp} and we find a
{\pp} $(k_2,l_2)$, which is not a multiple of $(a,b)$ such that 
$(k_1k_2,l_1l_2)$ is a {\pp}, which implies that $(k_2,l_2)$ is a {\dppp}.
Proceeding this way, we can construct an infinite family of {\dppps} which
are not multiples of each other.}
\end{rmk}

\begin{alg}{\algsty
The following algorithm decribes how to construct {\pps} $(k,l)$ from rational
points on $\Cab$ of infinite order.
\begin{itemize}
\item Let $P$ be a rational point on $\Cab$ of infinite order and
let $x_2$ be the $x$-coordinate of $2*P$.
\item Let $f$ and $g$ be relatively prime positive integers such
that $$\frac{g}{f}\;=\;\sqrt{x_2}\,.$$
\item Let $r$ and $s$ be relatively prime positive integers such
that $$\frac rs\;=\;\frac{bf+\sqrt{g^2+b^2f^2}}{g}\,.$$
\item Let $k:=r^2-s^2$ and let $l:=2rs$.
\end{itemize}
Then $(ak,bl)$ is a {\pp}. 
}
\end{alg}

%%%%%%%%%%%%%%%%%%%%%%%%%%%%%%%%%%%%%%%%%%%%%%%%%%%%%%%%%%

\bibliographystyle{plain}
%\bibliography{kubert}

\end{document}